\documentclass[11pt,reqno]{amsart}

\usepackage{amsmath, amsthm, amstext, amssymb, amsfonts, graphicx, bm, subfigure, fancyhdr}
\newtheorem{theorem}{Theorem}[section]
\newtheorem{corollary}[theorem]{Corollary}
\newtheorem{definition}[theorem]{Definition}

\newtheorem{proposition}[theorem]{Proposition}
\newtheorem*{remark}{Remark}

\numberwithin{equation}{section}

\newcommand{\PR}{\mathbb P}
\newcommand{\QR}{\mathbb Q}
\newcommand{\ER}{\mathbb E}
\newcommand{\CR}{\mathbb C}

\begin{document}

\title{Phase transitions in edge-weighted exponential\\ random graphs}

\fancyhf{}

\chead[Mei Yin]{Phase transitions in edge-weighted exponential
random graphs}

\cfoot{\thepage}

\author{Mei Yin}

\address{Department of Mathematics, University of Denver, Denver, CO 80208,
USA} \email {mei.yin@du.edu}

\dedicatory{\rm \today}

\begin{abstract}
The exponential family of random graphs represents an important
and challenging class of network models. Despite their
flexibility, conventionally used exponential random graphs have
one shortcoming. They cannot directly model weighted networks as
the underlying probability space consists of simple graphs only.
Since many substantively important networks are weighted, this
limitation is especially problematic. We extend the existing
exponential framework by proposing a generic common distribution
for the edge weights and rigorously analyze the associated phase
transitions and critical phenomena. We then apply these general
results to get concrete answers in exponential random graph models
where the edge weights are uniformly distributed.
\end{abstract}

\maketitle

\thispagestyle{empty}

\section{Introduction}
\label{intro} Exponential random graphs are among the most
widespread models for real-world networks as they are able to
capture a broad variety of common network tendencies by
representing a complex global structure through a set of tractable
local features. These rather general models are exponential
families of probability distributions over graphs, in which
dependence between the random edges is defined through certain
finite subgraphs, in imitation of the use of potential energy to
provide dependence between particle states in a grand canonical
ensemble of statistical physics. They are particularly useful when
one wants to construct models that resemble observed networks as
closely as possible, but without going into details of the
specific process underlying network formation. For history and a
review of developments, see e.g. Fienberg \cite{F1} \cite{F2},
Frank and Strauss \cite{FS}, H\"{a}ggstr\"{o}m and Jonasson
\cite{HJ}, Newman \cite{N}, and Wasserman and Faust \cite{WF}.

Many of the investigations into exponential random graphs employ
the elegant theory of graph limits as developed by Lov\'{a}sz and
coauthors (V.T. S\'{o}s, B. Szegedy, C. Borgs, J. Chayes, K.
Vesztergombi, ...) \cite{BCLSV1} \cite{BCLSV2} \cite{BCLSV3}
\cite{Lov} \cite{LS}. Building on earlier work of Aldous
\cite{Aldous} and Hoover \cite{Hoover}, the graph limit theory
connects sequences of graphs to a unified graphon space equipped
with a cut metric. Though the theory itself is tailored to dense
graphs, serious attempts have been made at formulating parallel
results for sparse graphs \cite{AL} \cite{BS} \cite{BCCZ1}
\cite{BCCZ2} \cite{CD2}. Since networks are often very large in
size, a pressing objective in studying exponential models is to
understand their asymptotic tendencies. From the point of view of
extremal combinatorics and statistical mechanics, emphasis has
been made on the variational principle of the limiting
normalization constant, concentration of the limiting probability
distribution, phase transitions and asymptotic structures. See
e.g. Aristoff and Zhu \cite{AZ}, Chatterjee and Diaconis
\cite{CD}, Kenyon et al. \cite{KRRS}, Kenyon and Yin \cite{KY},
Lubetzky and Zhao \cite{LZ1} \cite{LZ2}, Radin and Sadun
\cite{RS1} \cite{RS2}, and Radin and Yin \cite{RY}.

Despite their flexibility, conventionally used exponential random
graphs admittedly have one shortcoming. They cannot directly model
weighted networks as the underlying probability space consists of
simple graphs only, i.e., edges are either present or absent.
Since many substantively important networks are weighted, this
limitation is especially problematic. The need to extend the
existing exponential framework is thus justified, and several
generalizations have been proposed \cite{K} \cite{RPW}
\cite{WDBCD}. An alternative interpretation for simple graphs is
such that the edge weights are iid and satisfy a Bernoulli
distribution. This work will instead assume that the iid edge
weights follow a generic common distribution and rigorously
analyze the associated phase transitions and critical phenomena.

The rest of this paper is organized as follows. In Section
\ref{weight} we provide basics of graph limit theory and introduce
key features of edge-weighted exponential random graphs. In
Section \ref{statement} we summarize some important general
results, including a variational principle for the limiting
normalization constant (Theorems \ref{main1} and \ref{main2}) and
an associated concentration of measure (Theorem \ref{main3})
indicating that almost all large graphs lie near the maximizing
set. Theorems \ref{main4} and \ref{gen} then give simplified
versions of these theorems in the ``attractive'' region of the
parameter space where the parameters $\beta_2,...,\beta_k$ are all
non-negative. In Section \ref{app} we specialize to exponential
models where the edge weights are uniformly distributed and show
in Theorem \ref{phase} the existence of a first order phase
transition curve ending in a second order critical point. Lastly,
in Section \ref{discuss} we investigate the asymptotic phase
structure of a directed model where a large deviation principle is
missing.

\section{Edge-weighted exponential random graphs}
\label{weight} Let $G_n\in \mathcal{G}_n$ be an edge-weighted
undirected labeled graph on $n$ vertices, where the edge weights
$x_{ij}$ between vertex $i$ and vertex $j$ are iid real random
variables having a common distribution $\mu$. Any such graph
$G_n$, irrespective of the number of vertices, may be represented
as an element $h^{G_n}$ of a single abstract space $\mathcal{W}$
that consists of all symmetric measurable functions $h(x,y)$ from
$[0,1]^2$ into $\mathbb{R}$ (referred to as ``graph limits'' or
``graphons''), by setting $h^{G_n}(x,y)$ as the edge weight
between vertices $\lceil nx \rceil$ and $\lceil ny \rceil$ of
$G_n$. For a finite simple graph $H$ with vertex set
$V(H)=[k]=\{1,...,k\}$ and edge set $E(H)$ and a simple graph
$G_n$ on $n$ vertices, there is a notion of density of graph
homomorphisms, denoted by $t(H, G_n)$, which indicates the
probability that a random vertex map $V(H) \to V(G_n)$ is
edge-preserving,
\begin{equation}
\label{t} t(H, G_n)=\frac{|\text{hom}(H,
G_n)|}{|V(G_n)|^{|V(H)|}}.
\end{equation}
For a graphon $h\in \mathcal{W}$, define the graphon homomorphism
density
\begin{equation}
\label{tt} t(H, h)=\int_{[0,1]^k}\prod_{\{i,j\}\in E(H)}h(x_i,
x_j)dx_1\cdots dx_k.
\end{equation}
Then $t(H, G_n)=t(H, h^{G_n})$ by construction, and we take
(\ref{tt}) as the definition of graph homomorphism density $t(H,
G_n)$ for an edge-weighted graph $G_n$. This graphon
interpretation enables us to capture the notion of convergence in
terms of subgraph densities by an explicit ``cut distance'' on
$\mathcal{W}$:
\begin{equation}
d_{\square}(f, h)=\sup_{S, T \subseteq [0,1]}\left|\int_{S\times
T}\left(f(x, y)-h(x, y)\right)dx\,dy\right|
\end{equation}
for $f, h \in \mathcal{W}$. The common distribution $\mu$ for the
edge weights yields probability measure $\PR_n$ and the associated
expectation $\ER_n$ on $\mathcal{G}_n$, and further induces
probability measure $\QR_n$ on the space $\mathcal{W}$ under the
graphon representation.

A non-trivial complication is that the topology induced by the cut
metric is well defined only up to measure preserving
transformations of $[0,1]$ (and up to sets of Lebesgue measure
zero), which may be thought of vertex relabeling in the context of
finite graphs. To tackle this issue, an equivalence relation
$\sim$ is introduced in $\mathcal{W}$. We say that $f\sim h$ if
$f(x, y)=h_{\sigma}(x, y):=h(\sigma x, \sigma y)$ for some measure
preserving bijection $\sigma$ of $[0,1]$. Let $\tilde{h}$
(referred to as a ``reduced graphon'') denote the equivalence
class of $h$ in $(\mathcal{W}, d_{\square})$. Since $d_{\square}$
is invariant under $\sigma$, one can then define on the resulting
quotient space $\tilde{\mathcal{W}}$ the natural distance
$\delta_{\square}$ by $\delta_{\square}(\tilde{f},
\tilde{h})=\inf_{\sigma_1, \sigma_2}d_{\square}(f_{\sigma_1},
h_{\sigma_2})$, where the infimum ranges over all measure
preserving bijections $\sigma_1$ and $\sigma_2$, making
$(\tilde{\mathcal{W}}, \delta_{\square})$ into a metric space.
With some abuse of notation we also refer to $\delta_{\square}$ as
the ``cut distance''. After identifying graphs that are the same
after vertex relabeling, the probability measure $\PR_n$ yields
probability measure $\tilde{\PR}_n$ and the associated expectation
$\tilde{\ER}_n$ (which coincides with $\ER_n$). Correspondingly,
the probability measure $\QR_n$ induces probability measure
$\tilde{\QR}_n$ on the space $\tilde{\mathcal{W}}$ under the
measure preserving transformations.

By a $k$-parameter family of exponential random graphs we mean a
family of probability measures $\PR_n^{\beta}$ on $\mathcal{G}_n$
defined by, for $G_n\in\mathcal{G}_n$,
\begin{equation}
\label{pmf} \PR_n^{\beta}(G_n)=\exp\left(n^2\left(\beta_1
t(H_1,G_n)+\cdots+
  \beta_k t(H_k,G_n)-\psi_n^{\beta}\right)\right)\PR_n(G_n),
\end{equation}
where $\beta=(\beta_1,\dots,\beta_k)$ are $k$ real parameters,
$H_1,\dots,H_k$ are pre-chosen finite simple graphs (and we take
$H_1$ to be a single edge), $t(H_i, G_n)$ is the density of graph
homomorphisms, $\PR_n$ is the probability measure induced by the
common distribution $\mu$ for the edge weights, and
$\psi_n^{\beta}$ is the normalization constant,
\begin{equation}
\label{psi} \psi_n^{\beta}=\frac{1}{n^2}\log \ER_n
\left(\exp\left(n^2 \left(\beta_1 t(H_1,G_n)+\cdots+\beta_k
t(H_k,G_n)\right) \right)\right).
\end{equation}
We say that a phase transition occurs when the limiting
normalization constant $\displaystyle
\psi^\beta_\infty:=\lim_{n\to
  \infty}\psi_n^{\beta}$ has a singular point, as it is the
generating function for the limiting expectations of other random
variables,
\begin{equation}
\label{E} \lim_{n\to \infty}\ER_n^\beta t(H_i, G_n)=\lim_{n\to
\infty}\frac{\partial}{\partial
\beta_i}\psi_n^\beta=\frac{\partial}{\partial
\beta_i}\psi_\infty^\beta,
\end{equation}
\begin{equation}
\label{Cov} \lim_{n\to \infty}n^2\left(\CR\textrm{ov}_n^\beta
\left(t(H_i, G_n), t(H_j, G_n)\right)\right)=\lim_{n\to
\infty}\frac{\partial^2}{\partial \beta_i
\partial \beta_j}\psi_n^\beta=\frac{\partial^2}{\partial \beta_i
\partial \beta_j}\psi_\infty^\beta.
\end{equation}
The exchange of limits in (\ref{E}) and (\ref{Cov}) is nontrivial,
but may be justified using similar techniques as in Yang and Lee
\cite{YL}. Since homomorphism densities $t(H_i, G_n)$ are
preserved under vertex relabeling, the probability measure
$\tilde{\PR}_n^\beta$ and the associated expectation
$\tilde{\ER}_n^\beta$ (which coincides with $\ER_n^\beta$) may
likewise be defined.

\begin{definition}
A phase is a connected region of the parameter space $\{\beta\}$,
maximal for the condition that the limiting normalization constant
$\psi_\infty^\beta$ is analytic. There is a $j$th-order transition
at a boundary point of a phase if at least one $j$th-order partial
derivative of $\psi_\infty^\beta$ is discontinuous there, while
all lower order derivatives are continuous.
\end{definition}

More generally, we may consider exponential models where the terms
in the exponent defining the probability measure contain functions
on the graph space other than homomorphism densities. Let $T:
\tilde{\mathcal{W}} \rightarrow \mathbb{R}$ be a bounded
continuous function. Let the probability measure $\PR^T_n$ and the
normalization constant $\psi^T_n$ be defined as in (\ref{pmf}) and
(\ref{psi}), that is,
\begin{equation}
\label{pmf2}
\PR^T_n(G_n)=\exp\left(n^2(T(\tilde{h}^{G_n})-\psi^T_n)\right)\PR_n(G_n),
\end{equation}
\begin{equation}
\label{psi2} \psi^T_n=\frac{1}{n^2}\log \ER_n \exp\left(n^2
T(\tilde{h}^{G_n}) \right),
\end{equation}
Then the probability measure $\tilde{\PR}_n^T$ and the associated
expectation $\tilde{\ER}_n^T$ (which coincides with $\ER_n^T$) may
be defined in a similar manner.

We will assume that the common distribution $\mu$ on the edge
weights has \textit{finite support}, which implies that the
graphon space $\mathcal{W}$ under consideration is a finite subset
of $\mathbb{R}$. These $L^\infty$ graphons generalize graphons
that take values in $[0,1]$ only and are better suited for generic
edge-weighted graphs instead of just simple graphs. The ``finite
support'' assumption also assures that the moment generating
function $M(\theta)=\int e^{\theta x}\mu(dx)$ is finite for all
$\theta$ and the conjugate rate function of Cram\'{e}r, $I:
\mathbb{R}\rightarrow \mathbb{R}$, where
\begin{equation}
\label{I} I(x)=\sup_{\theta\in \mathbb{R}}\left(\theta x-\log
M(\theta)\right)
\end{equation}
is nicely defined. The domain of the function $I$ can be extended
to $\tilde{\mathcal{W}}$ in the usual manner:
\begin{equation}
\label{II} I(\tilde{h})=\int_{[0,1]^2}I(h(x,y))dxdy,
\end{equation}
where $h$ is any representative element of the equivalence class
$\tilde{h}$. It was shown in \cite{CV} that $I$ is well defined on
$\tilde{\mathcal{W}}$ and is lower semi-continuous under the cut
metric $\delta_\square$. The space $(\tilde{\mathcal{W}},
\delta_{\square})$ enjoys many other important properties that are
essential for the study of exponential random graph models. It is
a compact space and homomorphism densities $t(H, \cdot)$ are
continuous functions on it. In fact, homomorphism densities
characterize convergence under the cut metric: a sequence of
graphs converges if and only if its homomorphism densities
converge for all finite simple graphs, and the limiting
homomorphism densities then describe the resulting graphon.

\section{Statement of results}
\label{statement} The normalization constant plays a central role
in statistical mechanics because it encodes essential information
about the structure of the probability measure; even the existence
of its limit bears important consequences. In the case of
exponential random graphs, the computational tools currently used
by practitioners to compute this constant become unreliable for
large networks \cite{BBS} \cite{SPRH}. This problem however can be
circumvented if we know a priori that the limit of the
normalization constant exists. One can then choose a ``scaled
down'' network model with a smaller number of vertices and use the
exact value of the normalization constant in the scaled down model
as an approximation to the normalization constant in the larger
model, and a computer program that can evaluate the exact value of
the normalization constant for moderate sized networks would serve
the purpose \cite{Hunter}. The following Theorem \ref{main1} is a
generalization of the corresponding result (Theorem 3.1) in
Chatterjee and Diaconis \cite{CD}, where they assumed that the
edge weights $x_{ij}$ between vertices $i$ and $j$ are iid real
random variables satisfying a special common distribution --
Bernoulli having values $1$ and $0$ each with probability $1/2$.
For the sake of completeness and to motivate further discussions
in Theorem \ref{main2}, we present the proof details below.

\begin{theorem}
\label{main1} Let $T: \tilde{\mathcal{W}} \rightarrow \mathbb{R}$
be a bounded continuous function. Let $\psi_n^T$ and $I$ be
defined as before (see (\ref{psi2}), (\ref{I}) and (\ref{II})).
Then the limiting normalization constant $\displaystyle
\psi^T_\infty:=\lim_{n\rightarrow \infty}\psi_n^T$ exists, and is
given by
\begin{equation}
\label{setmax} \psi^T_\infty=\sup_{\tilde{h}\in
\tilde{\mathcal{W}}}\left(T(\tilde{h})-\frac{1}{2}I(\tilde{h})\right).
\end{equation}
\end{theorem}

\begin{proof}
For each Borel set $\tilde{A} \subseteq \tilde{\mathcal{W}}$ and
each $n$, define
\begin{equation}
\label{n} \tilde{A}^n=\{\tilde{h}\in \tilde{A}:
\tilde{h}=\tilde{h}^{G_n} \text{ for some } G_n\in
\mathcal{G}_n\}.
\end{equation}
Fix $\epsilon>0$. Since $T$ is a bounded function, there is a
finite set $R$ such that the intervals $\{(c, c +\epsilon): c\in
R\}$ cover the range of $T$. Since $T$ is a continuous function,
for each $c\in R$, the set $\tilde{U}_{c}$ consisting of reduced
graphons $\tilde{h}$ with $c<T(\tilde{h})<c+\epsilon$ is an open
set, and the set $\tilde{H}_{c}$ consisting of reduced graphons
$\tilde{h}$ with $c\leq T(\tilde{h})\leq c+\epsilon$ is a closed
set. Let $\tilde{U}_{c}^n$ and $\tilde{H}_{c}^n$ be respectively
defined as in (\ref{n}). The large deviation principle, Theorem
1.2 of \cite{CV}, implies that:
\begin{equation}
\limsup_{n\rightarrow \infty}\frac{\log
\tilde{\QR}_n(\tilde{H}_{c}^n)}{n^2}\leq -\inf_{\tilde{h}\in
\tilde{H}_{c}}\frac{1}{2}I(\tilde{h}),
\end{equation}
and that
\begin{equation}
\liminf_{n\rightarrow \infty}\frac{\log
\tilde{\QR}_n(\tilde{U}_{c}^n)}{n^2}\geq -\inf_{\tilde{h}\in
\tilde{U}_{c}}\frac{1}{2}I(\tilde{h}),
\end{equation}
where $\tilde{Q}_n$ denotes the family of probability measures on
$\tilde{\mathcal{W}}$ that is inherited from the corresponding
probability measures $\tilde{P}_n$ on $\mathcal{G}_n$.

We first consider $\limsup \psi_n^T$.
\begin{equation}
\exp(n^2\psi_n^T)\leq \sum_{c\in
R}\exp(n^2(c+\epsilon))\tilde{\QR}_n(\tilde{H}_{c}^n)\leq
|R|\sup_{c\in
R}\exp(n^2(c+\epsilon))\tilde{\QR}_n(\tilde{H}_{c}^n).
\end{equation}
This shows that
\begin{equation}\label{supce}
\limsup_{n\rightarrow \infty}\psi^{T}_{n}\leq \sup_{c\in
R}\left(c+\epsilon-\inf_{\tilde{h}\in
\tilde{H}_{c}}\frac{1}{2}I(\tilde{h})\right).
\end{equation}
Each $\tilde{h}\in \tilde{H}_{c}$ satisfies $T(\tilde{h})\geq c$.
Consequently,
\begin{equation}
\sup_{\tilde{h}\in
\tilde{H}_{c}}(T(\tilde{h})-\frac{1}{2}I(\tilde{h}))\geq
\sup_{\tilde{h}\in
\tilde{H}_{c}}(c-\frac{1}{2}I(\tilde{h}))=c-\inf_{\tilde{h}\in
\tilde{H}_{c}}\frac{1}{2}I(\tilde{h}).
\end{equation}
Substituting this in (\ref{supce}) gives
\begin{eqnarray}
\limsup_{n\rightarrow \infty}\psi^{T}_{n}&\leq&
\epsilon+\sup_{c\in R}\sup_{\tilde{h}\in
\tilde{H}_{c}}(T(\tilde{h})-\frac{1}{2}I(\tilde{h}))\\\notag&=&\epsilon+\sup_{\tilde{h}\in
\tilde{\mathcal{W}}}(T(\tilde{h})-\frac{1}{2}I(\tilde{h})).
\end{eqnarray}

Next we consider $\liminf \psi^{T}_{n}$.
\begin{equation}
\exp(n^2\psi^{T}_{n})\geq \sup_{c\in R}
\exp(n^2c)\tilde{\QR}_n(\tilde{U}_{c}^n).
\end{equation}
Therefore for each $c\in R$,
\begin{equation}\label{infc}
\liminf_{n \rightarrow \infty}\psi^{T}_{n}\geq
c-\inf_{\tilde{h}\in \tilde{U}_{c}}\frac{1}{2}I(\tilde{h}).
\end{equation}
Each $\tilde{h}\in \tilde{U}_{c}$ satisfies
$T(\tilde{h})<c+\epsilon$. Therefore
\begin{equation}
\sup_{\tilde{h}\in
\tilde{U}_{c}}(T(\tilde{h})-\frac{1}{2}I(\tilde{h}))\leq
\sup_{\tilde{h}\in
\tilde{U}_{c}}(c+\epsilon-\frac{1}{2}I(\tilde{h}))=c+\epsilon-\inf_{\tilde{h}\in
\tilde{U}_{c}}\frac{1}{2}I(\tilde{h}).
\end{equation}
Together with (\ref{infc}), this shows that
\begin{eqnarray}
\liminf_{n \rightarrow \infty}\psi^{T}_{n}&\geq&
-\epsilon+\sup_{c\in R}\sup_{\tilde{h}\in
\tilde{U}_{c}}(T(\tilde{h})-\frac{1}{2}I(\tilde{h}))\\\notag&=&-\epsilon+\sup_{\tilde{h}\in
\tilde{\mathcal{W}}}(T(\tilde{h})-\frac{1}{2}I(\tilde{h})).
\end{eqnarray}

Since $\epsilon$ is arbitrary, this completes the proof.
\end{proof}

When $T$ is not continuous, we lose control of the openness and
closedness of the sets in $\tilde{\mathcal{W}}$, and hence can not
resort to the large deviation principle of \cite{CV} to find an
explicit characterization of the limiting normalization constant
$\psi^T_\infty$, but we can still conclude that this important
limit exists under mild conditions.

\begin{theorem}
\label{main2} Let $T: \tilde{\mathcal{W}} \rightarrow \mathbb{R}$
be a bounded function. Let $I$ and $\psi_n^T$ be defined as before
(see (\ref{psi2}), (\ref{I}) and (\ref{II})). Then the limiting
normalization constant $\psi^T_\infty$ exists.
\end{theorem}

\begin{proof}
Since $T$ is a bounded function, there is a finite set $R$ such
that the intervals $\{(c, c +\epsilon): c\in R\}$ cover the range
of $T$. Similarly as in the proof of Theorem \ref{main1}, we have
\begin{equation}
\exp(n^2\psi_n^T)\leq |R|\sup_{c\in
R}\exp(n^2(c+\epsilon))\tilde{\QR}_n(\tilde{U}_{c}^n),
\end{equation}
and
\begin{equation}
\exp(n^2\psi^{T}_{n})\geq \sup_{c\in R}
\exp(n^2c)\tilde{\QR}_n(\tilde{U}_{c}^n).
\end{equation}
This implies that
\begin{equation}
\limsup_{n\rightarrow \infty}\psi^T_n-\liminf_{n\rightarrow
\infty}\psi^T_n\leq \frac{\log |R|}{n^2}+\epsilon.
\end{equation}
Since $\epsilon$ is arbitrary, this completes the proof.
\end{proof}

\begin{remark}
The boundedness condition on $T$ can be relaxed even further.
$(\log |R|)/(n^2) \rightarrow 0$ is sufficient.
\end{remark}

Let $\tilde{H}$ be the subset of $\tilde{\mathcal{W}}$ where
$\psi_\infty^T$ is maximized. By the compactness of
$\tilde{\mathcal{W}}$, the continuity of $T$ and the lower
semi-continuity of $I$, $\tilde{H}$ is a nonempty compact set. The
set $\tilde{H}$ encodes important information about the
exponential model (\ref{pmf2}) and helps to predict the behavior
of a typical random graph sampled from this model. The following
Theorem \ref{main3} states that in the large $n$ limit, the
quotient image $\tilde{h}^{G_n}$ of a random graph $G_n$ drawn
from (\ref{pmf2}) must lie close to $\tilde{H}$ with probability
$1$. Especially, if $\tilde{H}$ is a singleton set, the theorem
gives a law of large numbers for $G_n$; while if $\tilde{H}$ is
not a singleton set, the theorem points to the existence of a
first order phase transition.

\begin{theorem}
\label{main3} Let $\tilde{H}$ be defined as in the above
paragraph. Let $\PR_n^T$ (\ref{pmf2}) be the probability measure
on $\mathcal{G}_n$. Then
\begin{equation}
\lim_{n\rightarrow \infty}\delta_\square(\tilde{h}^{G_n},
\tilde{H})=0 \text{ almost surely}.
\end{equation}
\end{theorem}

\begin{proof}
The proof of this theorem follows a similar line of reasoning as
in the proof of the corresponding result (Theorem 3.2) in
Chatterjee and Diaconis \cite{CD}. The key observation is that the
distance between the graph and the maximizing set decays
exponentially as the number of vertices of the graph grows. For
any $\eta>0$ there exist $C, \gamma>0$ such that for all $n$ large
enough,
\begin{equation}
\PR_n^T(\delta_\square(\tilde{h}^{G_n}, \tilde{H})>\eta)\leq
Ce^{-n^2\gamma}.
\end{equation}
This is stronger than ordinary convergence in probability, which
only shows that the probability decays to zero but does not give
the speed. Utilizing the fast exponential decay rate, we can then
resort to probability estimates in the product space. By
Borel-Cantelli,
\begin{equation}
\sum_{n=1}^\infty \PR_n^T(\delta_\square(\tilde{h}^{G_n},
\tilde{H})>\eta)<\infty \text{ implies that }
\PR_n^T(\delta_\square(\tilde{h}^{G_n}, \tilde{H})>\eta \text{
infinitely often})=0.
\end{equation}
The conclusion hence follows.
\end{proof}

When the bounded continuous function $T$ on $\tilde{\mathcal{W}}$
is given by the sum of graph homomorphism densities
$T=\sum_{i=1}^k \beta_i t(H_i, \cdot)$, the statement of Theorems
\ref{main1} and \ref{main3} may be simplified in the
``attractive'' region of the parameter space where the parameters
$\beta_2,...,\beta_k$ are all non-negative, as seen in the
following Theorems \ref{main4} and \ref{gen}.

\begin{theorem}
\label{main4} Consider a general $k$-parameter exponential random
graph model (\ref{pmf}). Suppose $\beta_2,...,\beta_k$ are
non-negative. Then the limiting normalization constant
$\displaystyle \psi_\infty^\beta$ exists, and is given by
\begin{equation}
\label{lmax} \psi_{\infty}^{\beta}=\sup_{u}\left(\beta_1
u^{e(H_1)}+\cdots+\beta_k u^{e(H_k)}-\frac{1}{2}I(u)\right)\\,
\end{equation}
where $e(H_i)$ is the number of edges in $H_i$, $I$ is the
Cram\'{e}r function (\ref{I}) (\ref{II}), and the supremum is
taken over all $u$ in the domain of $I$, i.e., where $I<\infty$.
\end{theorem}

\begin{proof}
The proof of this theorem follows a similar line of reasoning as
in the proof of the corresponding result (Theorem 4.1) in
Chatterjee and Diaconis \cite{CD}, where the exact form of $I$ for
the Bernoulli distribution was given. The crucial step is
recognizing that $I$ is convex. So though the exact form of $I$ is
not always obtainable for a generic common distribution $\mu$, the
log of the moment generating function $M(\theta)$ is convex and
$I$, being the Legendre transform of a convex function, must also
be convex.
\end{proof}

\begin{theorem}
\label{gen} Let $G_n$ be an exponential random graph drawn from
(\ref{pmf}). Suppose $\beta_2,...,\beta_k$ are non-negative. Then
$G_n$ behaves like an Erd\H{o}s-R\'{e}nyi graph $G(n, u^*)$ in the
large n limit, where $u^*$ is picked randomly from the set $U$ of
maximizers of (\ref{lmax}).
\end{theorem}

\begin{proof}
The assertions in this theorem are direct consequences of Theorems
\ref{main3} and \ref{main4}.
\end{proof}

\section{An application: Uniformly distributed edge weights}
\label{app} Consider the set $\mathcal{G}_n$ of all edge-weighted
undirected labeled graphs on $n$ vertices, where the edge weights
$x_{ij}$ between vertex $i$ and vertex $j$ are iid real random
variables uniformly distributed on $(0,1)$. The common
distribution for the edge weights yields probability measure
$\PR_n$ and the associated expectation $\ER_n$ on $\mathcal{G}_n$.
Give the set of such graphs the probability
\begin{equation}
\PR_n^\beta(G_n)=\exp\left(n^2\left(\beta_1 t(H_1, G_n)+\beta_2
t(H_2, G_n)-\psi_n^\beta\right)\right)\PR_n(G_n),
\end{equation}
where $\beta=(\beta_1, \beta_2)$ are $2$ real parameters, $H_1$ is
a single edge, $H_2$ is a finite simple graph with $p\geq 2$
edges, and $\psi_n^\beta$ is the normalization constant,
\begin{equation}
\label{dpsi} \psi_n^\beta=\frac{1}{n^2}\log
\ER_n\left(\exp\left(n^2\left(\beta_1 t(H_1, G_n)+\beta_2 t(H_2,
G_n)\right)\right)\right).
\end{equation}
The associated expectation $\ER_n^\beta$ may be defined
accordingly, and a phase transition occurs when the limiting
normalization constant $\displaystyle \psi^\beta_\infty=\lim_{n\to
  \infty}\psi_n^{\beta}$ has a singular point.

\begin{theorem}
\label{phase} For any allowed $H_2$, the limiting normalization
constant $\psi_\infty^\beta$ is analytic at all $(\beta_1,
\beta_2)$ in the upper half-plane $(\beta_2\geq 0)$ except on a
certain decreasing curve $\beta_2=r(\beta_1)$ which includes the
endpoint $(\beta_1^c, \beta_2^c)$. The derivatives
$\frac{\partial}{\partial \beta_1}\psi_\infty^\beta$ and
$\frac{\partial}{\partial \beta_2}\psi_\infty^\beta$ have (jump)
discontinuities across the curve, except at the end point where,
however, all the second derivatives $\frac{\partial^2}{\partial
\beta_1^2}\psi_\infty^\beta$, $\frac{\partial^2}{\partial \beta_1
\partial \beta_2}\psi_\infty^\beta$ and $\frac{\partial^2}{\partial
\beta_2^2}\psi_\infty^\beta$ diverge.
\end{theorem}

\begin{corollary}
For any allowed $H_2$, the parameter space $\{(\beta_1, \beta_2):
\beta_2\geq 0\}$ consists of a single phase with a first order
phase transition across the indicated curve $\beta_2=r(\beta_1)$
and a second order phase transition at the critical point
$(\beta_1^c, \beta_2^c)$.
\end{corollary}

\begin{figure}
\centering
\includegraphics[clip=true, height=3.5in]{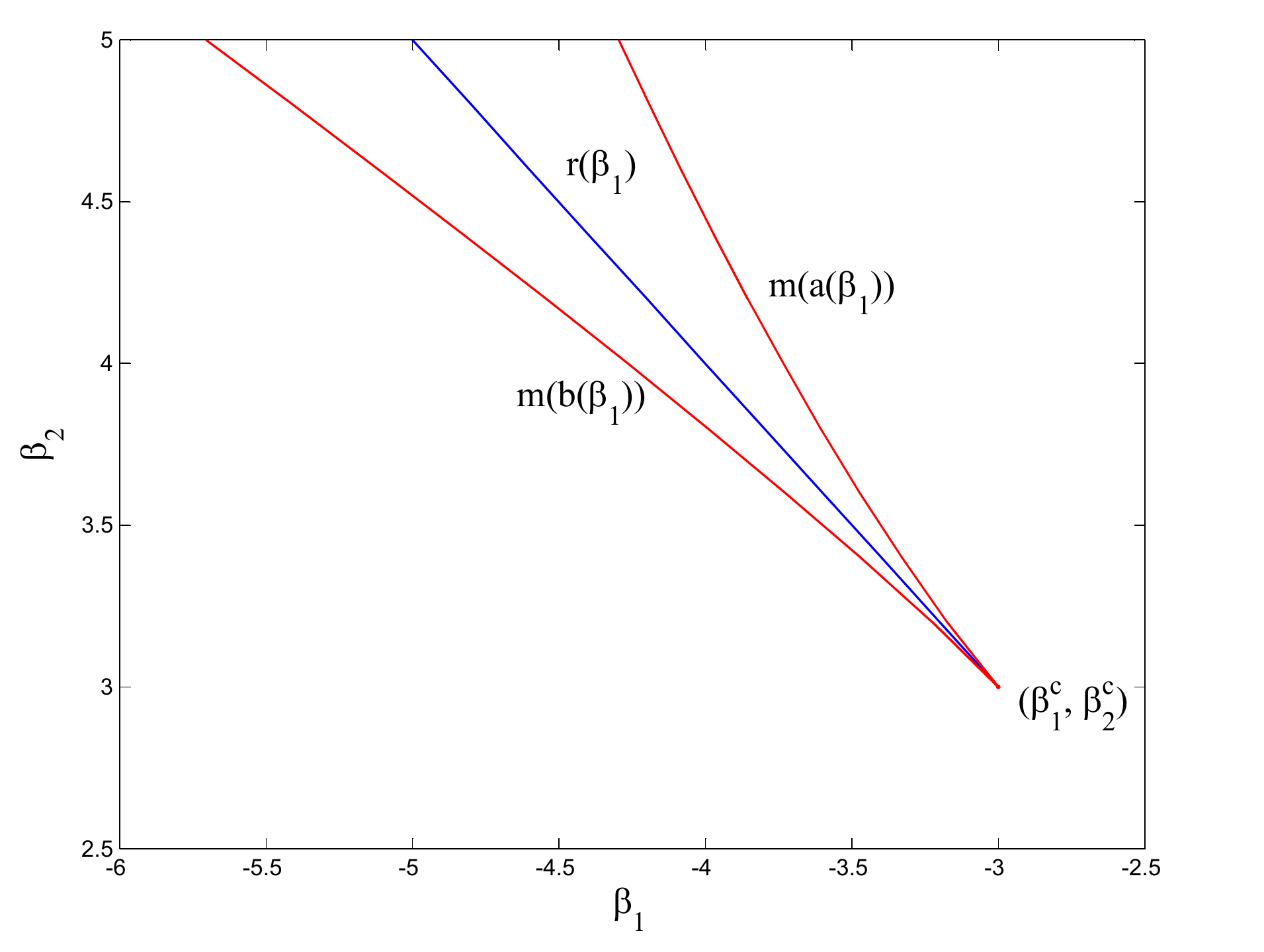}
\caption{The V-shaped region (with phase transition curve
$r(\beta_1)$ inside) in the $(\beta_1, \beta_2)$ plane. Graph
drawn for $p=2$.} \label{Vshape}
\end{figure}

\begin{proof}
The moment generating function for the uniform $(0,1)$
distribution is given by $M(\theta)=(e^\theta-1)/\theta$. The
associated Cram\'{e}r function $I$ is finite on $(0,1)$,
\begin{equation}
\label{Iu} I(u)=\sup_{\theta\in \mathbb{R}}\left(\theta u-\log
\frac{e^\theta-1}{\theta}\right),
\end{equation}
but does not admit a closed-form expression. As shown in Theorem
\ref{main4}, the limiting normalization constant
$\psi_\infty^\beta$ exists and is given by
\begin{equation}
\label{smax} \psi_\infty^\beta=\sup_{0\leq u \leq 1}
\left(\beta_1u+\beta_2u^p-\frac{1}{2}I(u)\right).
\end{equation}
A significant part of computing phase boundaries for the
$2$-parameter exponential model is then a detailed analysis of the
calculus problem (\ref{smax}). However, as straightforward as it
sounds, since the exact form of $I$ is not obtainable, getting a
clear picture of the asymptotic phase structure is not that easy
and various tricks need to be employed.

Consider the maximization problem for $l(u; \beta_1,
\beta_2)=\beta_1u+\beta_2u^p-\frac{1}{2}I(u)$ on the interval $[0,
1]$, where $-\infty<\beta_1<\infty$ and $0\leq \beta_2<\infty$ are
parameters. The location of maximizers of $l(u)$ on the interval
$[0, 1]$ is closely related to properties of its derivatives
$l'(u)$ and $l''(u)$,
\begin{equation}
l'(u)=\beta_1+p\beta_2u^{p-1}-\frac{1}{2}I'(u),
\end{equation}
\begin{equation*}
l''(u)=p(p-1)\beta_2u^{p-2}-\frac{1}{2}I''(u).
\end{equation*}

\begin{figure}
\centering
\includegraphics[clip=true, width=6in, height=2.5in]{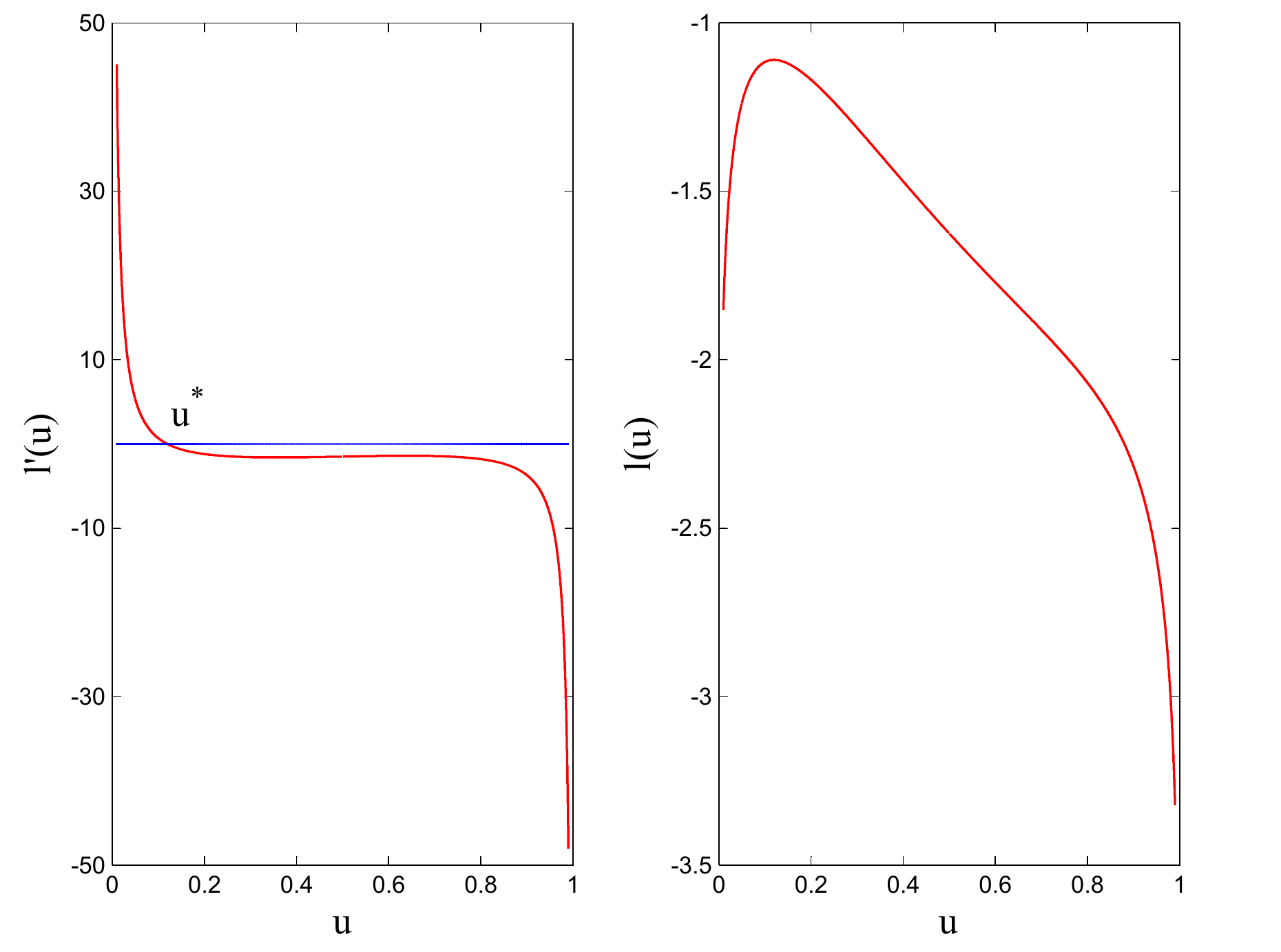}
\caption{Outside the V-shaped region, $l(u)$ has a unique local
maximizer (hence global maximizer) $u^*$. Graph drawn for
$\beta_1=-5$, $\beta_2=3.5$, and $p=2$.} \label{below}
\end{figure}

Following the duality principle for the Legendre transform between
$I(u)$ and $\log M(\theta)$ \cite{ZRM}, we first analyze
properties of $l''(u)$ on the interval $[0, 1]$. Recall that
\begin{equation}
I(u)+\log \frac{e^\theta-1}{\theta}=\theta u,
\end{equation}
where $\theta$ and $u$ are implicitly related. Taking derivatives,
we have
\begin{equation}
\label{dual} u=\left.\left(\log
\frac{e^\theta-1}{\theta}\right)\right|_\theta',
\end{equation}
\begin{equation*}
I''(u)\cdot \left.\left(\log
\frac{e^\theta-1}{\theta}\right)\right|_\theta''=1.
\end{equation*}
Consider the function
\begin{equation}
m(u)=\frac{I''(u)}{2p(p-1)u^{p-2}}
\end{equation}
on $[0, 1]$. By (\ref{dual}), analyzing properties of $m(u)$
translates to analyzing properties of the function
\begin{eqnarray}
n(\theta)&=&2p(p-1)\left.\left(\log
\frac{e^\theta-1}{\theta}\right)\right|_\theta''\cdot
\left(\left.\left(\log
\frac{e^\theta-1}{\theta}\right)\right|_\theta'\right)^{p-2}\\\notag
&=&2p(p-1)\left(\frac{e^{2\theta}-e^\theta(\theta^2+2)+1}{(e^\theta-1)^2\theta^2}\right)\left(\frac{e^\theta(\theta-1)+1}{(e^\theta-1)\theta}\right)^{p-2}
\end{eqnarray}
over $\mathbb{R}$, where $m(u)n(\theta)=1$ and $u$ and $\theta$
satisfy the dual relationship. We recognize that
\begin{equation}
\lim_{\theta\rightarrow -\infty}n(\theta)=0,
\end{equation}
\begin{equation*}
\lim_{\theta\rightarrow
0}n(\theta)=\frac{p(p-1)}{3}\left(\frac{1}{2}\right)^{p-1},
\end{equation*}
\begin{equation*}
\lim_{\theta\rightarrow \infty}n(\theta)=0,
\end{equation*}
which implies that $n(\theta)$ achieves a finite global maximum.
We claim that the global maximum can only be attained at
$\theta_0\geq 0$. This would further imply that there is a finite
global minimum for $m(u)$, and it can only be attained at $u_0\geq
\frac{1}{2}$. First suppose $p=2$. Then $n'(\theta_0)=0$ easily
implies that $\theta_0=0$. Now suppose $p>2$. Then since
$n'(\theta_0)=0$, we have
\begin{equation}
\label{compare}
\left.\left(\frac{e^{2\theta}-e^\theta(\theta^2+2)+1}{(e^\theta-1)^2\theta^2}\right)\right|_{\theta_0}'\left.\left(\frac{e^\theta(\theta-1)+1}{(e^\theta-1)\theta}\right)\right|_{\theta_0}
=-(p-2)\left.\left(\frac{e^{2\theta}-e^\theta(\theta^2+2)+1}{(e^\theta-1)^2\theta^2}\right)^2\right|_{\theta_0}<0.
\end{equation}
This says that $\theta_0$ is on the portion of $\mathbb{R}$ where
$\left(e^{2\theta}-e^\theta(\theta^2+2)+1\right)/\left((e^\theta-1)^2\theta^2\right)$
is decreasing, i.e., $\theta_0>0$. As $\theta$ goes from $0$ to
$\infty$, the left hand side of (\ref{compare}) first decreases
from $0$ then increases to $0$ and is always negative, whereas the
right hand side of (\ref{compare}) monotonically increases from
$-(p-2)/144$ to $0$ and approaches $0$ at a rate faster than the
left hand side, and so there exists a $\theta_0$ which makes both
sides equal. Our numerical computations show that $n'(\theta_0)=0$
is uniquely defined, and $n(\theta)$ increases from $-\infty$ to
$\theta_0$ and decreases from $\theta_0$ to $\infty$. See Table
\ref{table1} for some of these critical values.

\begin{table}
\begin{center}
\begin{tabular}{ccccc}
$p$ & $\theta_0$ & $n(\theta_0)$ & $u_0$ & $m(u_0)$ \\
\hline \hline \\
$2$ & $0$ & $0.3333$ & $0.5$ & $3$ \\
$3$ & $1.3251$ & $0.5575$ & $0.6073$ & $1.7937$ \\
$5$ & $2.9869$ & $0.8324$ & $0.7183$ & $1.2014$ \\
$10$ & $5.6256$ & $1.0894$ & $0.8259$ & $0.9180$ \\ \\
\end{tabular}
\end{center}
\caption{Critical values for $m(u)$ and $n(\theta)$ as a function
of $p$.} \label{table1}
\end{table}

The above analysis shows that for $\beta_2\leq m(u_0)$,
$l''(u)\leq 0$ over the entire interval $[0, 1]$; while for
$\beta_2>m(u_0)$, $l''(u)$ takes on both positive and negative
values, and we denote the transition points by $u_1$ and $u_2$
($u_1<u_0<u_2$). Properties of $l''(u)$ on $[0, 1]$ entails
properties of $l'(u)$ over the same interval. For $\beta_2 \leq
m(u_0)$, $l'(u)$ is monotonically decreasing. For
$\beta_2>m(u_0)$, $l'(u)$ is decreasing from $0$ to $u_1$,
increasing from $u_1$ to $u_2$, and then decreasing again from
$u_2$ to $1$.

The analytic properties of $l''(u)$ and $l'(u)$ help us
investigate properties of $l(u)$ on the interval $[0, 1]$. Being
the Legendre transform of a smooth function, $I(u)$ is a smooth
function, and grows unbounded when $u$ approaches $0$ or $1$. By
the duality principle for the Legendre transform, $I'(u)=\theta$
where $\theta$ and $u$ are linked through (\ref{dual}), so
$I'(0)=-\infty$ and $I'(1)=\infty$. This says that $l(u)$ is a
smooth function, $l(0)=l(1)=-\infty$, $l'(0)=\infty$ and
$l'(1)=-\infty$, so $l(u)$ can not be maximized at $0$ or $1$. For
$\beta_2\leq m(u_0)$, $l'(u)$ crosses the $u$-axis only once,
going from positive to negative. Thus $l(u)$ has a unique local
maximizer (hence global maximizer) $u^*$. For $\beta_2>m(u_0)$,
the situation is more complicated. If $l'(u_1)\geq 0$ (resp.
$l'(u_2)\leq 0$), $l(u)$ has a unique local maximizer (hence
global maximizer) at a point $u^*>u_2$ (resp. $u^*<u_1$). See
Figures \ref{below} and \ref{above}. If $l'(u_1)<0<l'(u_2)$, then
$l(u)$ has two local maximizers $u_1^*$ and $u_2^*$, with
$u_1^*<u_1<u_0<u_2<u_2^*$.

\begin{figure}
\centering
\includegraphics[clip=true, width=6in, height=2.5in]{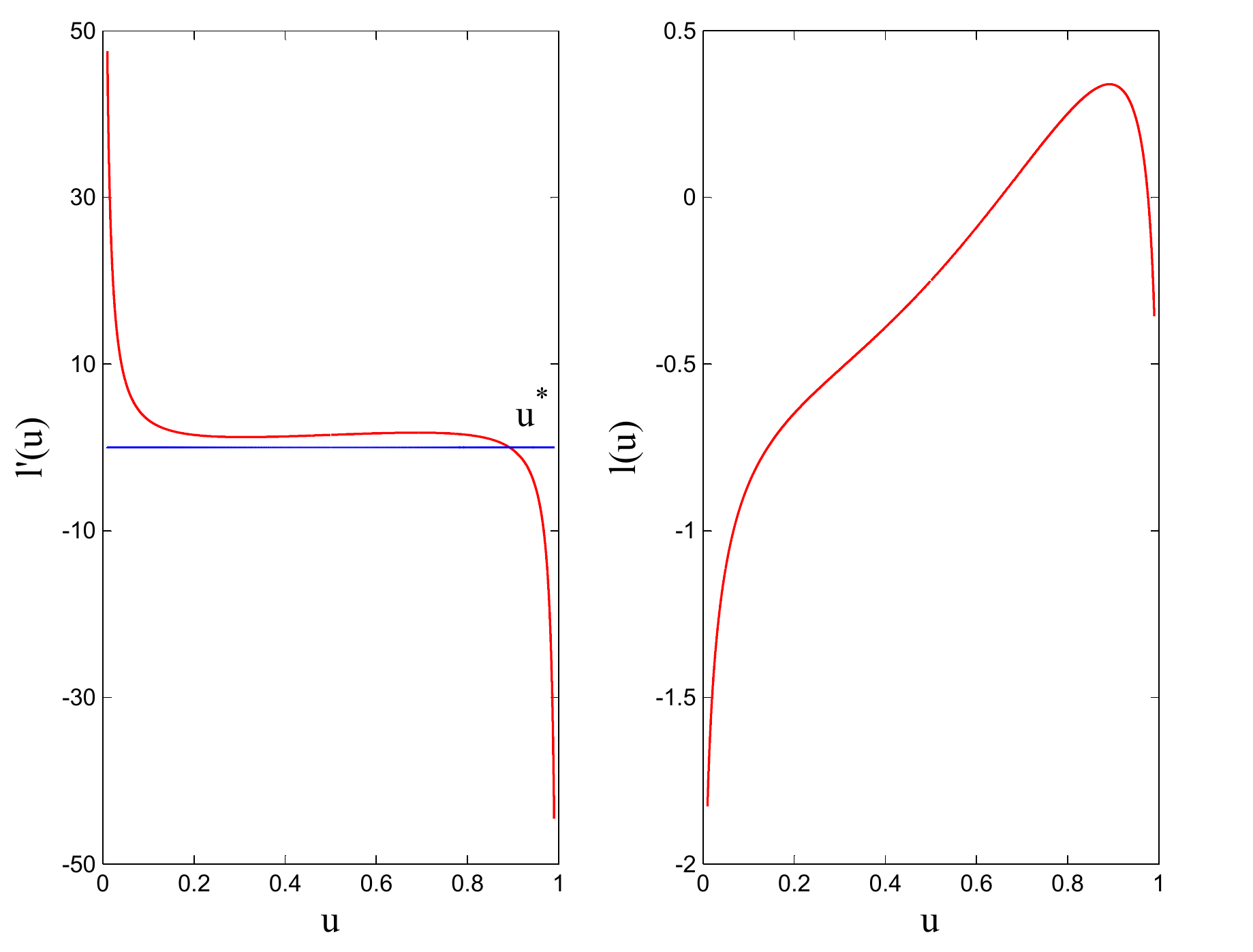}
\caption{Outside the V-shaped region, $l(u)$ has a unique local
maximizer (hence global maximizer) $u^*$. Graph drawn for
$\beta_1=-2.5$, $\beta_2=4$, and $p=2$.} \label{above}
\end{figure}

Let
\begin{equation}
f(u)=\frac{uI''(u)}{2(p-1)}-\frac{1}{2}I'(u)
\end{equation}
so that $l'(u_1)=\beta_1+f(u_1)$ and $l'(u_2)=\beta_1+f(u_2)$.
Again by the duality principle for the Legendre transform,
analyzing properties of the function $f(u)$ on $[0, 1]$ translates
to analyzing properties of the function
\begin{equation}
g(\theta)=\frac{\theta\left(e^\theta-1\right)\left(e^\theta(\theta-1)+1\right)}{2(p-1)\left(e^{2\theta}-e^\theta(\theta^2+2)+1\right)}-\frac{1}{2}\theta
\end{equation}
over $\mathbb{R}$, where $f(u)=g(\theta)$ and $u$ and $\theta$
satisfy the dual relationship. We recognize that
\begin{equation}
\lim_{\theta\rightarrow -\infty}g(\theta)=\infty,
\end{equation}
\begin{equation*}
\lim_{\theta\rightarrow 0}g(\theta)=\frac{3}{p-1},
\end{equation*}
\begin{equation*}
\lim_{\theta\rightarrow \infty}g(\theta)=\infty,
\end{equation*}
which implies that $g(\theta)$ achieves a finite global minimum.
We claim that the global minimum can only be attained at the same
unique $\theta_0$ that makes $n'(\theta)=0$. This can be checked
by identifying the condition for $g'(\theta)=0$. For $p=2$,
$g'(\theta_0)=0$ easily implies that $\theta_0=0$. For $p>2$,
since $g'(\theta_0)=0$, we have
\begin{multline}
\left.\left(\frac{e^\theta(\theta-1)+1}{(e^\theta-1)\theta}\right)\right|_{\theta_0}'\left.\left(\frac{e^{2\theta}-e^\theta(\theta^2+2)+1}{(e^\theta-1)^2\theta^2}\right)\right|_{\theta_0}-
\left.\left(\frac{e^{2\theta}-e^\theta(\theta^2+2)+1}{(e^\theta-1)^2\theta^2}\right)\right|_{\theta_0}'\left.\left(\frac{e^\theta(\theta-1)+1}{(e^\theta-1)\theta}\right)\right|_{\theta_0}
\\=(p-1)\left.\left(\frac{e^{2\theta}-e^\theta(\theta^2+2)+1}{(e^\theta-1)^2\theta^2}\right)^2\right|_{\theta_0},
\end{multline}
which yields the same solution as (\ref{compare}) after a simple
transformation. Thus $g(\theta)$ decreases from $-\infty$ to
$\theta_0$ and increases from $\theta_0$ to $\infty$. This further
implies that there is a finite global minimum for $f(u)$ attained
at $u_0$, and $f(u)$ decreases from $0$ to $u_0$ and increases
from $u_0$ to $1$. Both $\theta_0$ and $u_0$ coincide with the
critical values listed in Table \ref{table1}. See Table
\ref{table2}. We conclude that $l'(u_1)\geq 0$ for $\beta_1\geq
-f(u_0)$ and $\beta_2=m(u_0)$. The only possible region in the
$(\beta_1, \beta_2)$ plane where $l'(u_1)<0<l'(u_2)$ is bounded by
$\beta_1<-f(u_0):=\beta_1^c$ and $\beta_2>m(u_0):=\beta_2^c$.

\begin{table}
\begin{center}
\begin{tabular}{ccccc}
$p$ & $\theta_0$ & $g(\theta_0)$ & $u_0$ & $f(u_0)$ \\
\hline \hline \\
$2$ & $0$ & $3$ & $0.5$ & $3$ \\
$3$ & $1.3251$ & $1.3222$ & $0.6073$ & $1.3222$ \\
$5$ & $2.9869$ & $0.1059$ & $0.7183$ & $0.1059$ \\
$10$ & $5.6256$ & $-1.1723$ & $0.8259$ & $-1.1723$ \\ \\
\end{tabular}
\end{center}
\caption{Critical values for $f(u)$ and $g(\theta)$ as a function
of $p$.} \label{table2}
\end{table}

We examine the behavior of $l'(u_1)$ and $l'(u_2)$ more closely
when $\beta_1$ and $\beta_2$ are chosen from this region. Recall
that $u_1<u_0<u_2$. By monotonicity of $f(u)$ on the intervals
$(0, u_0)$ and $(u_0, 1)$, there exist continuous functions
$a(\beta_1)$ and $b(\beta_1)$ of $\beta_1$, such that $l'(u_1)<0$
for $u_1>a(\beta_1)$ and $l'(u_2)>0$ for $u_2>b(\beta_1)$. As
$\beta_1\rightarrow -\infty$, $a(\beta_1)\rightarrow 0$ and
$b(\beta_1)\rightarrow 1$. $a(\beta_1)$ is an increasing function
of $\beta_1$, whereas $b(\beta_1)$ is a decreasing function, and
they satisfy $f(a(\beta_1))=f(b(\beta_1))=-\beta_1$. The
restrictions on $u_1$ and $u_2$ yield restrictions on $\beta_2$,
and we have $l'(u_1)<0$ for $\beta_2<m(a(\beta_1))$ and
$l'(u_2)>0$ for $\beta_2>m(b(\beta_1))$. As $\beta_1\rightarrow
-\infty$, $m(a(\beta_1))\rightarrow \infty$ and
$m(b(\beta_1))\rightarrow \infty$. $m(a(\beta_1))$ and
$m(b(\beta_1))$ are both decreasing functions of $\beta_1$, and
they satisfy $l'(u_1)=0$ when $\beta_2=m(a(\beta_1))$ and
$l'(u_2)=0$ when $\beta_2=m(b(\beta_1))$. As $l'(u_2)>l'(u_1)$ for
every $(\beta_1, \beta_2)$, the curve $m(b(\beta_1))$ must lie
below the curve $m(a(\beta_1))$, and together they generate the
bounding curves of the $V$-shaped region in the $(\beta_1,
\beta_2)$ plane with corner point $(\beta_1^c, \beta_2^c)$ where
two local maximizers exist for $l(u)$. See Figures \ref{Vshape},
\ref{lower} and \ref{upper}.

\begin{figure}
\centering
\includegraphics[clip=true, width=6in, height=2.5in]{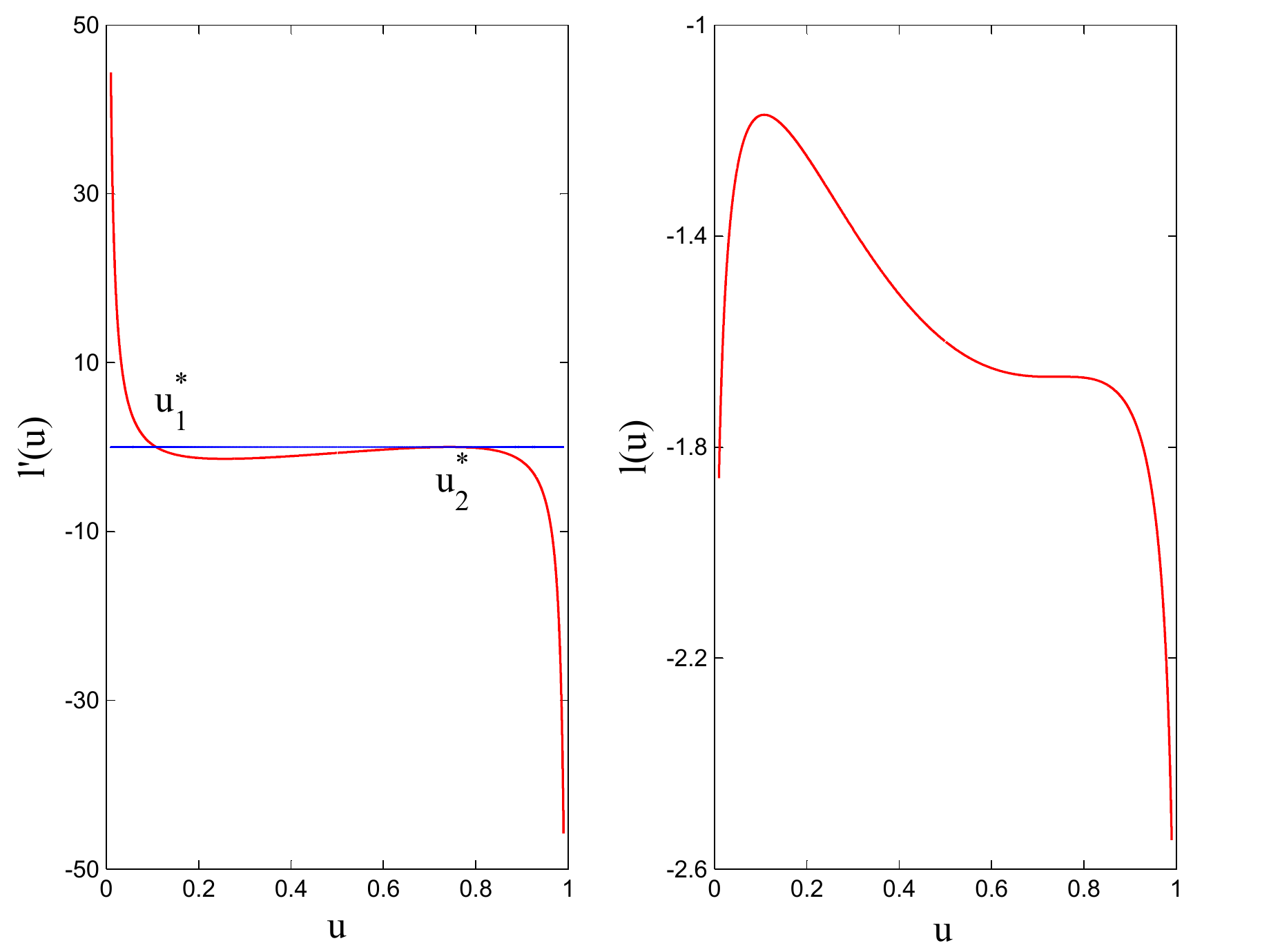}
\caption{Along the lower bounding curve $m(b(\beta_1))$ of the
V-shaped region, $l'(u)$ has two zeros $u_1^*$ and $u_2^*$, but
only $u_1^*$ is the global maximizer for $l(u)$. Graph drawn for
$\beta_1=-5.7$, $\beta_2=5$, and $p=2$.} \label{lower}
\end{figure}

Fix an arbitrary $\beta_1<\beta_1^c$, we examine the effect of
varying $\beta_2$ on the graph of $l'(u)$. It is clear that
$l'(u)$ shifts upward as $\beta_2$ increases and downward as
$\beta_2$ decreases. As a result, as $\beta_2$ gets large, the
positive area bounded by the curve $l'(u)$ increases, whereas the
negative area decreases. By the fundamental theorem of calculus,
the difference between the positive and negative areas is the
difference between $l(u_2^*)$ and $l(u_1^*)$, which goes from
negative ($l'(u_2)=0$, $u_1^*$ is the global maximizer) to
positive ($l'(u_1)=0$, $u_2^*$ is the global maximizer) as
$\beta_2$ goes from $m(b(\beta_1))$ to $m(a(\beta_1))$. Thus there
must be a unique $\beta_2$: $m(b(\beta_1))<\beta_2<m(a(\beta_1))$
such that $u_1^*$ and $u_2^*$ are both global maximizers, and we
denote this $\beta_2$ by $r(\beta_1)$. See Figures \ref{Vshape}
and \ref{along}. The parameter values of $(\beta_1, r(\beta_1))$
are exactly the ones for which positive and negative areas bounded
by $l'(u)$ equal each other. An increase in $\beta_1$ induces an
upward shift of $l'(u)$, and may be balanced by a decrease in
$\beta_2$. Similarly, a decrease in $\beta_1$ induces a downward
shift of $l'(u)$, and may be balanced by an increase in $\beta_2$.
This justifies that $r(\beta_1)$ is monotonically decreasing in
$\beta_1$.

The rest of the proof follows as in the proof of the corresponding
result (Theorem 2.1) in Radin and Yin \cite{RY}, where some
probability estimates were used. A (jump) discontinuity in the
first derivatives of $\psi_\infty^\beta$ across the curve
$\beta_2=r(\beta_1)$ indicates a discontinuity in the expected
local densities (\ref{E}), while the divergence of the second
derivatives of $\psi_\infty^\beta$ at the critical point
$(\beta_1^c, \beta_2^c)$ implies that the covariances of the local
densities go to zero more slowly than $1/n^2$ (\ref{Cov}). We omit
the proof details.
\end{proof}

\begin{remark}
The maximization problem (\ref{smax}) is solved at a unique value
$u^*$ off the phase transition curve $\beta_2=r(\beta_1)$, and at
two values $u_1^*$ and $u_2^*$ along the curve. As
$\beta_1\rightarrow -\infty$ (resp. $\beta_2\rightarrow \infty$),
$u_1^*\rightarrow 0$ and $u_2^*\rightarrow 1$. The jump from
$u_1^*$ to $u_2^*$ is quite noticeable even for small parameter
values of $\beta$. For example, taking $p=2$, $\beta_1=-5$, and
$\beta_2=5$, numerical computations yield that $u_1^* \approx
0.137$ and $u_2^* \approx 0.863$ with $l(u_1^*)\approx l(u_2^*)
\approx -1.0854$.
\end{remark}

\begin{figure}
\centering
\includegraphics[clip=true, width=6in, height=2.5in]{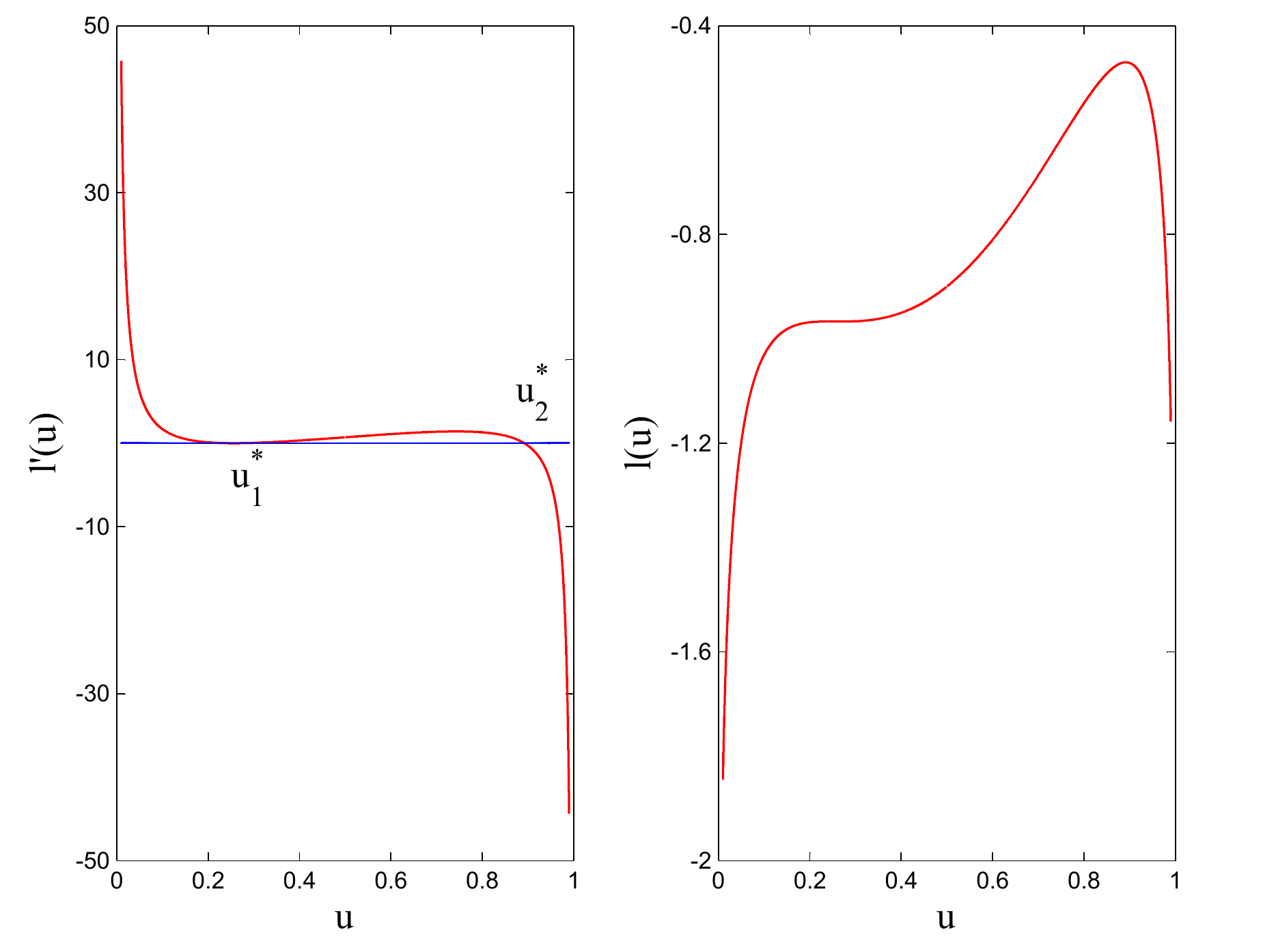}
\caption{Along the upper bounding curve $m(a(\beta_1))$ of the
V-shaped region, $l'(u)$ has two zeros $u_1^*$ and $u_2^*$, but
only $u_2^*$ is the global maximizer for $l(u)$. Graph drawn for
$\beta_1=-4.3$, $\beta_2=5$, and $p=2$.} \label{upper}
\end{figure}

\begin{proposition}
The Cram\'{e}r function $I(u)$ (\ref{Iu}) for the uniform
distribution on $(0, 1)$ is symmetric about the line
$u=\frac{1}{2}$.
\end{proposition}

\begin{proof}
Recall that $I(u)$ is the Legendre transform of the log moment
generating function $\log M(\theta)$ for the uniform $(0, 1)$
distribution. Take $0<u<\frac{1}{2}$. Though the exact form of $I$
is not obtainable, following the duality principle, we have
\begin{equation}
I(u)=\theta_1 u-\log\frac{e^{\theta_1}-1}{\theta_1},
\end{equation}
\begin{equation*}
I(1-u)=\theta_2 (1-u)-\log\frac{e^{\theta_2}-1}{\theta_2},
\end{equation*}
where $\theta_1$ and $\theta_2$ are both finite and uniquely
defined, and satisfy $I'(u)=\theta_1$ and $I'(1-u)=\theta_2$. We
show that $I(u)=I(1-u)$ when $\theta_2=-\theta_1$. The calculation
is straightforward:
\begin{equation}
I(1-u)=-\theta_1 (1-u)-\log\frac{e^{-\theta_1}-1}{-\theta_1}
=\theta_1 u-\log\frac{e^{\theta_1}-1}{\theta_1}=I(u).
\end{equation}
This verifies our claim.
\end{proof}

\begin{corollary}
Take $H_1$ a single edge and $H_2$ a finite simple graph with
$p\geq 2$ edges. The associated phase transition curve
$\beta_2=r(\beta_1)$ lies above the straight line
$\beta_2=-\beta_1$ when $p\geq 3$, and is exactly portion of the
straight line $\beta_2=-\beta_1$ ($\beta_1\leq -3$) when $p=2$.
\end{corollary}

\begin{proof}
From the proof of Theorem \ref{phase}, there are two global
maximizers $u_1^*$ and $u_2^*$ for $l(u)$ along the phase
transition curve $\beta_2=r(\beta_1)$, with $u_0\geq \frac{1}{2}$
and $0<u_1^*<u_0<u_2^*<1$. Furthermore, the $y$-coordinate
$\beta_2^c$ of the critical point $(\beta_1^c, \beta_2^c)$ is
always positive. On the straight line $\beta_1+\beta_2=0$, we
rewrite $l(u)=\beta_1(u-u^p)-\frac{1}{2}I(u)$. First suppose
$p=2$. Since $I(u)$ and $u-u^2$ are both symmetric about the line
$u=\frac{1}{2}$, two global maximizers $u_1^*$ and $u_2^*$ exist
for $l(u)$. Next consider the generic case $p>2$. Analytical
calculations give that $u-u^p<(1-u)-(1-u)^p$ for
$0<u<\frac{1}{2}$. Since $I(u)$ is symmetric, this says that for
$\beta_1<0$ (resp. $\beta_2>0$), the global maximizer $u^*$ of
$l(u)$ satisfies $u^*\leq \frac{1}{2}$. This implies the desired
result.
\end{proof}

\begin{figure}
\centering
\includegraphics[clip=true, width=6in, height=2.5in]{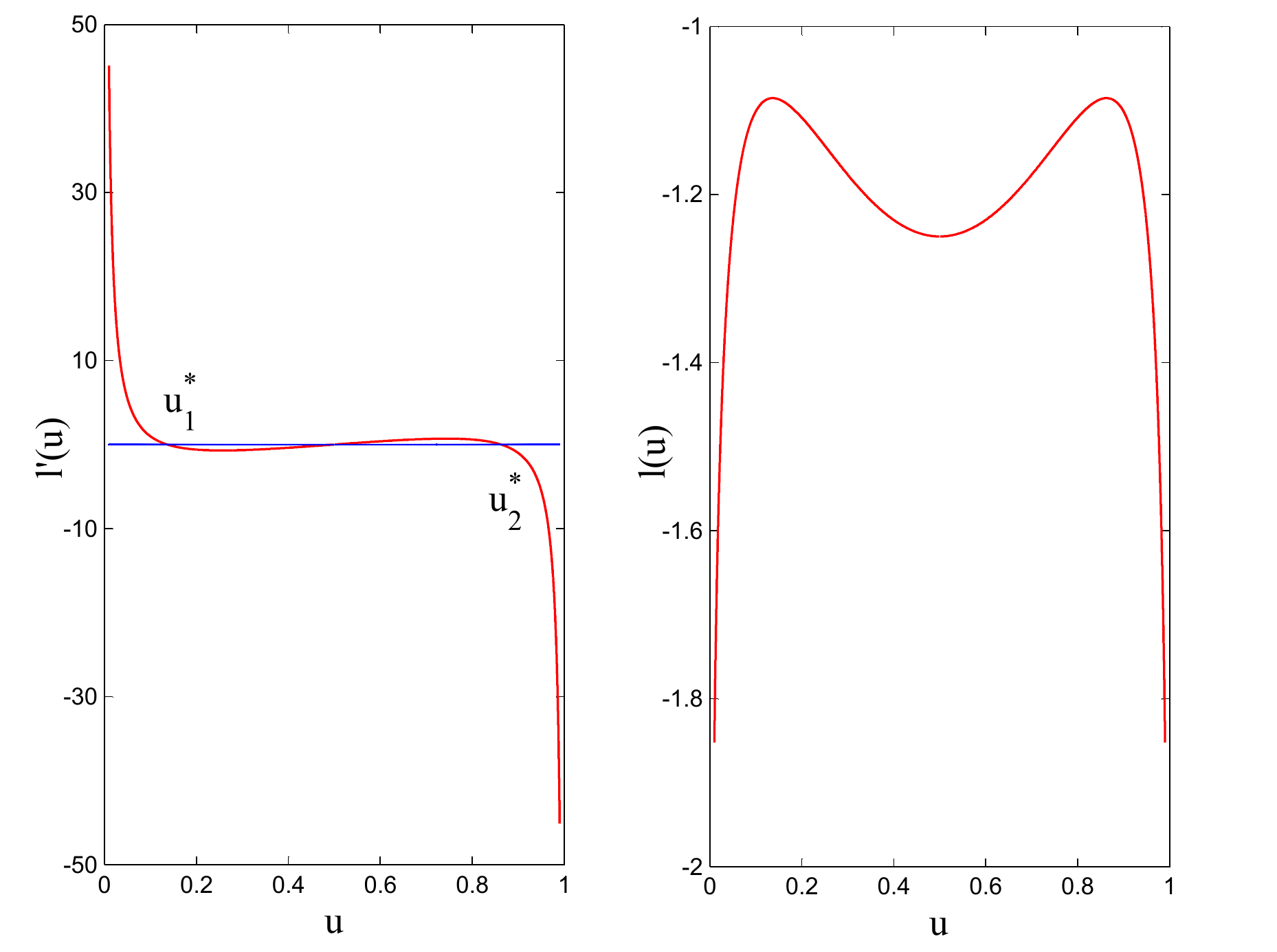}
\caption{Along the phase transition curve $r(\beta_1)$, $l(u)$ has
two local maximizers $u_1^*$ and $u_2^*$, and both are global
maximizers for $l(u)$. Graph drawn for $\beta_1=-5$, $\beta_2=5$,
and $p=2$.} \label{along}
\end{figure}

\section{Further discussion}
\label{discuss} But what if the common distribution $\mu$ on the
edge weights does not have finite support? Unfortunately, we do
not have a generic large deviation principle in this case. As an
example, we will examine the asymptotic phase structure for a
completely solvable exponential model and hope the analysis would
provide some inspiration. Let $G_n\in \mathcal{G}_n$ be an
edge-weighted directed labeled graph on $n$ vertices, where the
edge weights $x_{ij}$ from vertex $i$ to vertex $j$ are iid real
random variables whose common distribution $\mu$ is standard
Gaussian. As in the undirected case, the common distribution for
the edge weights yields probability measure $\PR_n$ and the
associated expectation $\ER_n$ on $\mathcal{G}_n$. Give the set of
such graphs the probability
\begin{equation}
\PR_n^\beta(G_n)=\exp\left(n^2\left(\beta_1e(G_n)+\beta_2s(G_n)-\psi_n^\beta\right)\right)\PR_n(G_n),
\end{equation}
where $\beta=(\beta_1, \beta_2)$ are $2$ real parameters, $H_1$ is
a directed edge, $H_2$ is a directed $2$-star, $e(G_n)$ and
$s(G_n)$ are respectively the directed edge and $2$-star
homomorphism densities of $G_n$,
\begin{equation}
e(G_n)=\frac{1}{n^2}\sum_{1\leq i,j\leq n} x_{ij}, \hspace{0.2cm}
s(G_n)=\frac{1}{n^3}\sum_{1\leq i,j,k \leq n} x_{ij}x_{ik},
\end{equation}
and $\psi_n^\beta$ is the normalization constant,
\begin{equation}
\label{dpsi} \psi_n^\beta=\frac{1}{n^2}\log
\ER_n\left(\exp\left(n^2\left(\beta_1e(G_n)+\beta_2s(G_n)\right)\right)\right).
\end{equation}
The associated expectation $\ER_n^\beta$ may be defined
accordingly, and a phase transition occurs when the limiting
normalization constant $\displaystyle \psi^\beta_\infty=\lim_{n\to
  \infty}\psi_n^{\beta}$ has a singular point.

Plugging the formulas for $e(G_n)$ and $s(G_n)$ into (\ref{dpsi})
and using the iid property of the edge weights, we have
\begin{eqnarray}
\psi_n^\beta&=&\frac{1}{n^2}\log
\ER_n\left(\exp\left(\beta_1\sum_{i=1}^n\left(\sum_{j=1}^n
x_{ij}\right)+\frac{\beta_2}{n}\sum_{i=1}^n\left(\sum_{j=1}^n
x_{ij}\right)^2\right)\right)\\
&=&\frac{1}{n}\log\ER \left(\exp\left(\beta_1
Y+\frac{\beta_2}{n}Y^2\right)\right),\notag
\end{eqnarray}
where $Y=\sum_{j=1}^n x_{1j}$ satisfies a Gaussian distribution
with mean $0$ and variance $n$, and $\ER$ is the associated
expectation. We compute, for $\beta_2<\frac{1}{2}$,
\begin{equation}
\ER \left(\exp\left(\beta_1
Y+\frac{\beta_2}{n}Y^2\right)\right)=\int_{-\infty}^\infty
e^{\beta_1y+\frac{\beta_2}{n}y^2}\frac{1}{\sqrt{2\pi
n}}e^{-\frac{y^2}{2n}}dy
=\frac{1}{\sqrt{1-2\beta_2}}e^{\frac{n\beta_1^2}{2(1-2\beta_2)}},
\end{equation}
which implies that
\begin{equation}
\psi_\infty^\beta=\frac{\beta_1^2}{2(1-2\beta_2)}.
\end{equation}
This is a smooth function in terms of the parameters $\beta_1$ and
$\beta_2$, and so $\psi_\infty^\beta$ does not admit a phase
transition.

\section*{Acknowledgements}
This work originated at the Special Session on Topics in
Probability at the 2016 AMS Western Spring Sectional Meeting,
organized by Tom Alberts and Arjun Krishnan. Mei Yin's research
was partially supported by NSF grant DMS-1308333. She thanks Sean
O'Rourke and Lingjiong Zhu for helpful conversations.

\end{document}